\documentclass[a4paper,UKenglish,cleveref, autoref,nolineno, thm-restate]{socg-lipics-v2019}

\usepackage{mathtools}
\usepackage{latexsym,color,graphicx,amsmath,amssymb,amsthm}
\usepackage{babel}

\def\reals{{\mathbb R}}

\def\Inov{{I_{\rm nov}}} 
\def\Inovp{{I'_{\rm nov}}} 
\def\A{{\cal A}}

\title{On rich lenses in planar arrangements of circles and related problems}

\author{Esther Ezra}%
{School of Computer Science, Bar Ilan University, Ramat Gan, Israel}%
{ezraest@cs.biu.ac.il}%
{https://orcid.org/0000-0001-8133-1335}%
{Work partially supported by NSF CAREER under grant CCF:AF-1553354 %
and by Grant 824/17 from the Israel Science Foundation.}

\author{Orit E. Raz}
{Institute of Mathematics, Hebrew University, Jerusalem, Israel}
{oritraz@mail.huji.ac.il}
{https://orcid.org/0000-0002-2910-436X}%
{}

\author{Micha Sharir}
{School of Computer Science, Tel Aviv University, Tel Aviv Israel}
{michas@tauex.tau.ac.il}
{http://orcid.org/0000-0002-2541-3763}
{Work partially supported by ISF Grant 260/18, by grant 1367/2016
from the German-Israeli Science Foundation (GIF), and by
Blavatnik Research Fund in Computer Science at Tel Aviv University.}

\author{Joshua Zahl}
{Department of Mathematics, University of British Columbia, Vancouver, B.C., Canada}
{jzahl@math.ubc.ca}
{http://orcid.org/0000-0001-5129-8300}
{Work supported by an NSERC Discovery Grant.}

\titlerunning{On rich lenses in arrangements of circles} 

\authorrunning{E. Ezra, O.~E. Raz, M. Sharir and J. Zahl} 

\Copyright{Esther Ezra and Orit E. Raz and Micha Sharir and Joshua Zahl} 



\category{} 

\relatedversion{} 

\supplement{}

\hideLIPIcs  

\nolinenumbers 

\EventEditors{John Q. Open and Joan R. Access}
\EventNoEds{2}
\EventLongTitle{42nd Conference on Very Important Topics (CVIT 2016)}
\EventShortTitle{SoCG 2021}
\EventAcronym{SoCG}
\EventYear{2021}
\EventDate{December 24--27, 2016}
\EventLocation{Little Whinging, United Kingdom}
\EventLogo{}
\SeriesVolume{42}
\ArticleNo{XX}

\begin{document}
 
\maketitle

\begin{abstract}
We show that the maximum number of pairwise non-overlapping $k$-rich lenses 
(lenses formed by at least $k$ circles) in an arrangement of $n$ circles in 
the plane is $O\left(\frac{n^{3/2}\log{(n/k^3)}}{k^{5/2}} + \frac{n}{k} \right)$, 
and the sum of the degrees of the lenses of such a family 
(where the degree of a lens is the number of circles that form it)
is $O\left(\frac{n^{3/2}\log{(n/k^3)}}{k^{3/2}} + n\right)$. 
Two independent proofs of these bounds are given, each interesting in its 
own right (so we believe).
We then show that these bounds lead to the known bound of 
\cite{ANPPSS,MT} on the number of point-circle incidences in the plane. 
Extensions to families of more general algebraic curves and some other related
problems are also considered.
\end{abstract}

\section{Introduction}

Let $C$ be a set of circles in the plane.
A \emph{lens} in the arrangement $\A(C)$ consists of a pair of distinct points $p,q$ and a set of circles $C'\subset C$, each of which contain
$p$ and $q$. We will denote a lens by $\lambda_{p,q}(C')$. We say that two lenses $\lambda_{p,q}(C')$ and $\lambda_{s,t}(C'')$ are 
\emph{overlapping} if there is a circle $c\in C'\cap C''$ so that the shorter arc of $c$ 
containing $p$ and $q$ intersects the shorter arc of $c$ containing $s$ and $t$.\footnote{%
  This definition requires a small modification if either $p,q$ or $s,t$ are antipodal points of $c$; if $p,q$ are antipodal points
  of $c$, then overlapping means that $s$ and $t$ are contained in opposite half-circles.} 
If two lenses are not overlapping, we call them \emph{non-overlapping}. 
Finally, the \emph{degree} of a lens $\lambda_{p,q}(C')$ is the cardinality of $C'$, 
and we say a lens is \emph{$k$-rich} if it has degree at least $k$.

In this paper, we will be concerned with bounding the maximum size of a collection of pairwise non-overlapping 
$k$-rich lenses determined by a set of $n$ circles in the plane. As we will see below, this question is closely 
related to the problem of \emph{lens cutting}, which has a host of applications in combinatorial geometry; 
chief among these is the problem of obtaining incidence bounds for points and circles in the plane. 

In \cite{MT} (sharpening a bound earlier obtained in \cite{ANPPSS,ArS:lens}) Marcus and Tardos proved 
that if $C$ is a set of $n$ circles, then any set of pairwise non-overlapping 2-rich lenses in $C$ has cardinality 
$O(n^{3/2}\log n)$. Using standard random sampling techniques, this implies that any set of pairwise 
non-overlapping $k$-rich lenses has cardinality $O\big(\frac{n^{3/2}\log(n/k)}{k^{3/2}}\big)$. 
Our main result considerably improves this bound. 

\begin{theorem} \label{krich52-lens}
Let $C$ be a set of $n$ circles in the plane, let $k\geq 2$, and let $\Lambda$ be a set of pairwise 
non-overlapping $k$-rich lenses. Then $|\Lambda|=O\left(\frac{n^{3/2}\log{(n/k^3)}}{k^{5/2}} + \frac{n}{k} \right)$, 
and the sum of the degrees of the lenses in $\Lambda$ is 
$O\left(\frac{n^{3/2}\log{(n/k^3)}}{k^{3/2}} + n \right)$. 
\end{theorem}
 
As mentioned, Theorem \ref{krich52-lens} was proved for the case $k=2$ by Marcus and Tardos \cite{MT}. 
When $k$ is large, we will show that Theorem \ref{krich52-lens} can be recast as an incidence problem 
between points and lines in $\reals^3$. Crucially, we will show that only a few incidences of
the type we analyze can occur inside any plane, and this will allow us to use a variant of 
Guth and Katz's point-line incidence bound from \cite{GK} to prove Theorem \ref{krich52-lens} 
when $k\geq n^{1/3}$. The details of this argument will be discussed in Section \ref{sec:prelim}. 

For intermediate values of $k$, we will give two proofs of Theorem \ref{krich52-lens}. The first proof 
yields the bounds stated above, and the second proof gives a sightly weaker bound, in which the $\log{(n/k^3)}$ 
factor is weakened to $\operatorname{polylog} n$. Both of these proofs will use polynomial partitioning 
to divide the arrangement $C$ of circles into smaller sub-arrangements. This breaks the problem of 
estimating $|\Lambda|$ into many smaller sub-problems, with smaller corresponding parameters $n'$ 
(for the number of circles) and $k'$ (for the richness). 
In the first proof, we will construct our partitioning so that in each sub-problem we have $k'=2$, while 
in the second proof we will construct our partitioning so that in each sub-problem we have $k'=(n')^{1/3}$. 

\medskip
\noindent\emph{Remarks.}\\
1. When $k\ge n^{1/3}\log^{2/3}n$, Theorem \ref{krich52-lens} states that $|\Lambda|=O(n/k)$. 
This bound is tight, since we can choose $C$ to be a union of $n/k$ sets of circles, where each 
set of circles has cardinality $k$ and the circles in each set contain a common pair of points. 
For smaller values of $k$ we conjecture that the bound in Theorem \ref{krich52-lens} is not tight.

\medskip
\noindent 2. Theorem~\ref{krich52-lens} implies that the circles in $C$ can be cut into 
$O\left( \frac{n^{3/2}\log {(n/k^3)}}{k^{3/2}} + n\right)$ arcs, so that no pair of points is contained 
in $k$ of the arcs, i.e., the resulting collection of arcs do not form any $k$-rich lens. 
Combining this observation with a variant of Sz\'ekely's crossing-lemma technique~\cite{Sze}, 
yields a new proof that the number of incidences between $m$ points and $n$ circles in the plane is 
(see Section~\ref{sec:incidence_bounds}):
\[
O\left( m^{2/3}n^{2/3} + m^{6/11}n^{9/11}\log^{2/11}n + m + n \right).
\]
This bound was first proved in \cite{ANPPSS,MT}. The point-circle incidence problem is among the 
most basic problems in incidence geometry, and has been studied intensively during the first half 
of the 2000's~\cite{ANPPSS,ArS:lens,MT}, culminating in the bound above.
This bound is strongly suspected not to be tight for $n^{1/3} \le m < n^{5/4}\log^{3/2}n$
(which is the range where the second term dominates), but no improvement has
been found in the last 15 years. While our result also does not yield 
an improvement, it provides a new proof (two proofs as a matter of fact),
and we hope that this development will spur efforts to improve the above bound. 

\section{Preliminaries: The case of large or small $k$} \label{sec:prelim}

In this section we will prove Theorem \ref{krich52-lens} when $k$ is small or $k\geq n^{1/3}$. 
As discussed above, when $k$ is small (smaller than some constant) then the result immediately follows from \cite{MT}.
\begin{theorem}[Marcus and Tardos \cite{MT}]\label{thm:MT}
Let $C$ be a set of $n$ circles in the plane, let $k\geq 2$, and let $\Lambda$ be a set of 
pairwise non-overlapping $k$-rich lenses in $C$. Then $\Lambda$ has cardinality $O(n^{3/2}\log n)$, 
and the sum of the degrees of the lenses in $\Lambda$ is also $O(n^{3/2}\log n)$.
\end{theorem}

When $k\geq n^{1/3}$, Theorem \ref{krich52-lens} will follow from a variant of Guth and Katz's 
point-line incidence bound \cite{GK}. Before stating this result we will need to introduce some 
additional notation. In what follows, $C$ will be a set of circles in the plane, $k\geq 2$, and 
$\Lambda$ will be a set of pairwise non-overlapping $k$-rich lenses in $C$. We define $\deg(\Lambda)$ 
to be the sum of the degrees of the lenses in $\Lambda$. We say that a circle $c\in C$ 
\emph{participates} in a lens $\lambda_{p,q}(C')\in\Lambda$ if $c\in C'$. 

We identify each circle $c$, with center $(x,y)$ and radius $r$, with the point 
\[
c^* = (x,y,r^2-x^2-y^2)
\]
in $\reals^3$. We define $C^* = \{c^*\mid c\in C\}$, and
identify each point $p=(p_x,p_y)\in\reals^2$ with the plane 
\[ 
p^* = \{(x,y,z)\mid z = -2p_xx - 2p_yy + (p_x^2+p_y^2)\}.
\] 
Observe that $(x,y,r)\in p^*$ if and only if $(x-p_x)^2 + (y-p_y)^2 = r^2$, i.e. the point $p$ 
is contained in the circle centered at $(x,y)$ of radius $r$. We identify each lens 
$\lambda = \lambda_{p,q}(C')$ with the line $\ell_\lambda = p^*\cap q^*$. 
Note that the lines $p^*\cap q^*$ and $s^*\cap t^*$ coincide if and only if $\{p,q\} = \{s,t\}$ (and therefore our setting does not contain coinciding lines).\footnote{%
  This is because all planes of the form $p^*$ are tangent to the
  paraboloid $\Pi:\; z = -x^2-y^2$. A line $p^*\cap q^*$ that
  is disjoint from $\Pi$ is contained in exactly two such tangent planes, and these planes determine $p$ and $q$.}
Define
\[
  L(\Lambda) = \{ \ell_{\lambda}\mid \lambda\in\Lambda \}.
\]
For technical reasons, it will be convenient to require that no two distinct lenses in $\Lambda$ share the same pair $\{p,q\}$ of endpoints, and thus the map $\lambda_{p,q}(C') \mapsto \ell_{p,q}$ is injective on $\Lambda$, i.e. 
$|L(\Lambda)|=|\Lambda|$. As we will see, this additional assumption is harmless, and we will discuss it briefly in Remark \ref{rem:disjointPtsAssumption} below.

We define
\[
\Inov(\Lambda)  = \{ (c^*,\ell_{\lambda})\mid c\in C,\ \lambda\in\Lambda,\ c\ \text{participates in}\ \lambda  \}.
\] 
If the sets $C$ and $\Lambda$ are apparent from the context, we write $\Inov$ in place of $\Inov(\Lambda)$. 
Note that $|\Inov| = \deg(\Lambda)$. If $(c^*,\ell_{\lambda})\in \Inov$ then $c^*\in\ell_{\lambda}$. 
Thus $\Inov\subset I(C^*,L(\Lambda))$, where $I(C^*,L(\Lambda))$ denotes the set of all incidences
between the points of $C^*$ and the lines of $L(\Lambda)$. 
Note that $\Inov $ may be a proper subset of $I(C^*,L(\Lambda))$ due to the pairwise
non-overlapping property of the lenses in $\Lambda$. That is, a circle $c$ might 
pass through the vertices $p$, $q$ of $\lambda$ but it does not participate
in $\lambda$ because there is another lens $\lambda'\in \Lambda $, in which $c$ 
does participate, so that the arc of $c$ in $\lambda'$ overlaps its arc 
between $p$ and $q$; see Figure~\ref{overlens2} for an illustration.  

\begin{figure}[htb]
  \begin{center}
    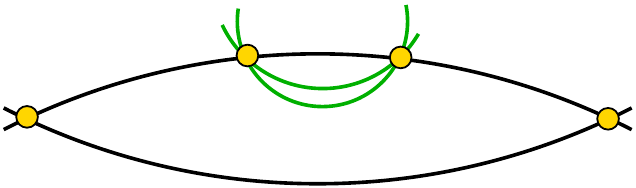
    \caption{The non-overlapping property may exclude some arcs from participating in a lens.} 
    \label{overlens2}
  \end{center}
\end{figure}

As a first attempt to bound $|\Inov|$, observe that the Szemer\'edi-Trotter theorem \cite{ST} implies that 
$|\Inov|=O(n^{2/3} |\Lambda|^{2/3} + n + |\Lambda|)$, and thus, since each $\lambda\in\Lambda$ 
is $k$-rich, we have, when $k$ is sufficiently large, $|\Lambda|=O(n^2/k^3 + n/k)$. Unfortunately 
this bound is too weak to prove Theorem \ref{krich52-lens}. To obtain a stronger bound, we will 
use the following crucial property about non-overlapping lenses, which implies that few of the incidences in 
$\Inov$ can concentrate in a plane. 

\begin{lemma}\label{no3coplanar}
  Let $C$ be a set of circles and let $\Lambda$ be a set of pairwise non-overlapping lenses in $C$. 
  Let $\lambda_1,\lambda_2,\lambda_3\in \Lambda$ be distinct lenses, and suppose that there is a 
  circle $c\in C$ that participates in all three lenses, that is, $(c^*,\ell_{\lambda_i})\in \Inov$ 
  for $i=1,2,3$. Then $\ell_{\lambda_1},\ell_{\lambda_2},\ell_{\lambda_3}$ are not coplanar.
\end{lemma}
\begin{proof}
The transformations $c\mapsto c^*$ and $\lambda\mapsto\ell_\lambda$ described above have the following property. 
For each nonvertical plane $h\subset\reals^3$ there exists a point $w\in\reals^2$ and a \emph{power}\footnote{%
  Recall that the power of a point $w$ with respect to a circle $c$, 
  centered at $\xi$ and having radius $r$, is $|w\xi|^2-r^2$.}
$\pi$ such that 
\[
h = \{ c^*\mid c\ \textrm{is a circle, and}\ w\ \textrm{has power}\ \pi\ \textrm{with respect to}\ c\}.
\]
Next, suppose that $\ell_{\lambda_1},\ell_{\lambda_2},\ell_{\lambda_3}$ lie in a common 
nonvertical plane $h$ (this plane must necessarily contain $c^*$), and let $w$ and $\pi$ be the point 
and power associated with $h$. Let $p_i$ and $q_i$ be the vertices of $\lambda_i$,
for $i=1,2,3$. It is then easy to see that $w$ must lie on each of the lines
(in the $xy$-plane) through $p_i$ and $q_i$, for $i=1,2,3$, and the power 
$\pi$ is $\pm |wp_i|\cdot|wq_i|$, where the sign is positive (resp.~negative) 
if $w$ lies outside (resp.~inside) the segment $p_iq_i$ (any other point cannot 
have a fixed power with respect to all the circles $c$ whose dual points $c^*$ lie on $\ell_{\lambda_i}$).
In other words, the lines through $p_1q_1$, $p_2q_2$, and $p_3q_3$ are concurrent
and meet at $w$. This however is impossible, because the circle $c$ participates
in all three lenses, which implies, as is easily verified (see Figure~\ref{power}), 
that at least two of the lenses $\lambda_1$, $\lambda_2$, $\lambda_3$ must be 
overlapping, a contradiction that completes the proof for nonvertical planes.

\begin{figure}[htb]
  \begin{center}
    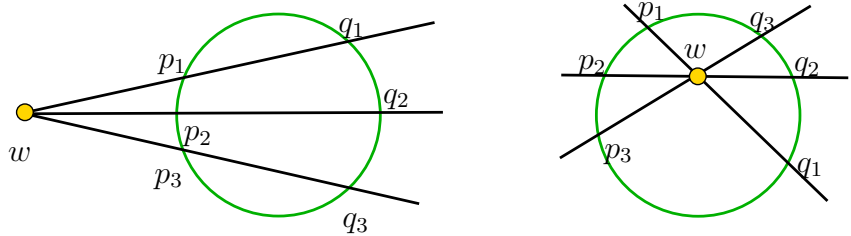
    \caption{The lines in $\reals^3$ corresponding to three pairwise non-overlapping lenses that share
      a common circle, which participates in all three of them, cannot be coplanar. (They are not
      the lines in the $xy$-plane drawn in the figure.)} 
    \label{power}
  \end{center}
\end{figure}

The situation is similar when $h$ is vertical. In this case all the circles
$c$ for which $c^*\in h$ are centered at points on the line $\ell$ of intersection
of $h$ with the $xy$-plane. For a lens $\lambda = \lambda_{p,q}$, the associated line  $\ell_\lambda$ is contained in $h$ if and only if $\ell$ is the bisector of $pq$.
Again, the fact that the lenses of $\Lambda$ are pairwise non-overlapping is easily
seen to imply that a circle $c$ can contain at most two pairs $p,q$ such
that $\lambda_{p,q}$ is a lens in $L$, with $\ell_\lambda\subset h$,
and $c$ participates in $\lambda_{p,q}$. Hence, 
$\ell_{\lambda_1},\ell_{\lambda_2},\ell_{\lambda_3}$ cannot all lie in $h$.
\end{proof}

Note that a point $c^*$ can be incident to arbitrarily many lines on a plane $h$, but
Lemma~\ref{no3coplanar} implies that only two of them (at most) can contribute to $\Inov$.

\begin{corollary} \label{cor:fewIncidencesInPlane}
  Let $C$ be a set of circles and let $\Lambda$ be a set of pairwise non-overlapping lenses in $C$. 
  Let $h\subset\reals^3$ be a plane, and let $(C^*)'\subset C^*$ and $L'\subset L(\Lambda)$ be the 
  set of points and lines contained in $h$, respectively. Then 
  \[ 
  |\Inov(\Lambda) \cap I((C^*)',\ L')|\leq 2|(C^*)'|.
  \]
\end{corollary}

Corollary \ref{cor:fewIncidencesInPlane} suggests that lines contained in a plane contribute 
few incidences to $\Inov$. Later in our arguments, we will need to study the contribution to 
$\Inov$ coming from lines contained in an algebraic surface. The following lemma says that 
after removing a small number of ill-behaved lines, the contribution to $\Inov$ is still small. 

\begin{lemma} \label{lem:incidencesInZeroSet}
  Let $C$ be a set of circles and let $\Lambda$ be a set of pairwise non-overlapping lenses in $C$. 
  Let $P\in\reals[x,y,z]$ be a polynomial of degree $D$, and let $L'\subset L(\Lambda)$ be a set of 
  lines contained in the zero set $Z(P)$ of $P$. Then there is a set $L''\subset L'$ of cardinality 
  at most $11D^2$, so that
  \[
  |\Inov(\Lambda) \cap I(C^*\cap Z(P), L'\backslash L'')| \leq 2|C^*\cap Z(P)| + D|L'\backslash L''|. 
  \]
\end{lemma}
\begin{proof}
This result follows immediately from the statements in Guth and Katz~\cite[Section 3]{GK}, so we just 
briefly sketch the proof. Write $Z(P)=Z^{(1)}\cup Z^{(2)}\cup Z^{(3)}\cup Z^{(4)}$, where 
$Z^{(1)}$ is a union of planes, $Z^{(2)}$ is a union of reguli, $Z^{(3)}$ is a union of 
irreducible surfaces that are singly ruled by lines and are not planes or reguli, and $Z^{(4)}$ is the union 
of all irreducible components of $Z(P)$ that are not ruled. Let $L''$ be the set of lines 
contained in $Z^{(4)}$. By \cite[Corollary 3.3]{GK} we have $|L''|\leq 11 D^2$. 

Let $Z_1,\ldots,Z_h$ be the irreducible components of $Z(P)$. For each index $i=1,\ldots,h$, 
let $C^*_i$ be the set of points $c^*\in C^*$ that are contained in $Z_i$ and are not contained 
in any $Z_j$ with $j<i$. Similarly, let $L_i$ be the set of lines $\ell\in L'\backslash L''$ that 
are contained in $Z_i$ and are not contained in any $Z_j$ with $j<i$ 
(note that if the component $Z_i$ is not ruled, then by definition $L_i$ is empty).

First, we count the number of incidences $(c^*,\ell) \in I(C^*\cap Z(P), L'\backslash L'')$ 
for which $c^*\in C_i$ and $\ell\in L_j$ with $j\neq i$. For such an incidence, we must have that 
$\ell$ properly intersects $Z_i$. Thus there are at most $(D-1)|L'\backslash L''|$ incidences of this form. 

Next we count the number of incidences $(c^*,\ell) \in I(C^*\cap Z(P), L'\backslash L'')$ for which 
$c^*\in C_i$ and $\ell\in L_i$. If $Z_i$ is a plane, then by Corollary \ref{cor:fewIncidencesInPlane}, 
there are $\leq 2|C^*_i|$ incidences of this type. If $Z_i$ is a regulus, and hence doubly ruled, then 
it immediately follows that there are at most $2|C^*_i|$ incidences of this type. Finally, if $Z_i$ 
is singly ruled, then $Z_i$ has at most one exceptional point (incident to infinitely many lines
contained in $Z_i$), and at most two exceptional (non-generator) lines,
in the terminology of \cite{GK}, which then implies that there are at most $2|C^*_i|+|L_i|$ 
incidences of this type. Summing the above contributions, we conclude that
\[
|\Inov \cap I(C^*\cap Z(P), L'\backslash L'')| \leq (D-1)|L'\backslash L''| + \sum_i (2|C^*_i|+|L_i|) =
2|C^*\cap Z(P)| + D|L'\backslash L''|.
\] 
\end{proof}

\begin{proposition}\label{prop:ptLineGK}
  Let $C$ be a set of circles and let $\Lambda$ be a set of pairwise non-overlapping lenses in $C$. 
  Then there is an absolute constant $A$ so that
  \begin{equation}\label{eq:boundOnINov}
    |\Inov(\Lambda)| \leq A \big( |C|^{1/2}|\Lambda|^{3/4} + |C| + |\Lambda| \big).
  \end{equation}
\end{proposition}
\begin{proof}
The proposition is a slight variant of Guth and Katz's point-line incidence bound from \cite{GK}, 
so we just briefly sketch the proof. We prove the result by induction on $|\Lambda|$.  Let $M = |C|$, 
let $L = L(\Lambda),$ and let $N=|L|$. First we can suppose that $N\leq M^2$. If not, then 
Proposition \ref{prop:ptLineGK} follows immediately from the K\H{o}v\'ari-S\'os-Tur\'an theorem 
(see~\cite{AgPa}), because the incidence graph of the points and the lines does not contain $K_{2,2}$ as a subgraph.

Let $D = \lfloor \min\left\{ M^{1/2}N^{-1/4},\ N^{1/2}/10\right\}\rfloor$. 
We can suppose that $N^{1/2}/10$ (and thus $D$) is at least one, since otherwise 
$|\Inov(\Lambda)|\leq 10|C|$ and we are done. Using the polynomial partitioning for varieties 
established by Guth~\cite{Guth}, we can find a polynomial $P\in\reals[x,y,z]$ of degree $\le D$ so that 
$\reals^3\backslash Z(P)$ is a union of $O(D^3)$ open connected sets (such sets are often 
called \emph{cells}), so that each cell contains $O(M/D^3)$ points from $C^*$, and each cell 
is intersected by $O(N/D^2)$ lines from $L$. If $D = N^{1/2}/10$ then each cell intersects $O(1)$ 
lines from $L$. Since each point from $C^*$ is contained in at most one cell, we have in this case
\[
I( C^* \backslash Z(P) ,\Lambda) = O(M).
\]
If $D=M^{1/2}N^{-1/4}$, then standard incidence estimates allow us to bound 
\[
I( C^* \backslash Z(P) ,\Lambda) = O(M^{1/2}N^{3/4}).
\]
Similarly, standard incidence estimates allow us to bound
\[ 
| \{ (c^*,\ell)\in I( C^* \cap Z(P) ,\Lambda)\mid \ell\not\subset Z(P) \}| = O(ND) = O(M^{1/2}N^{3/4}).
\]
Let $L'\subset L$ be the set of lines contained in $Z(P)$. 
Applying Lemma \ref{lem:incidencesInZeroSet}, we obtain a set $L''\subset L'$ with 
$|L''|\leq 11D^2\leq |L|/2$, and
\[
|\Inov \cap I(C^*\cap Z(P), L'\backslash L'')| \leq 2|C^*\cap Z(P)| + D|L'\backslash L''| = O(M^{1/2}N^{3/4} + M). 
\]
Finally, we apply the induction hypothesis to bound
\[
|\Inov \cap I(C^*\cap Z(P), L'')| \leq A( M^{1/2} |L''|^{3/4} + M + |L''|) \leq 2^{-3/4}AM^{1/2}N^{3/4} + A(M + N).
\]
Combining these bounds, we conclude that
\[
|\Inov| \leq 2^{-3/4}AM^{1/2}N^{3/4} + A(M+N) + O(M^{1/2}N^{3/4}),
\]
where the implicit constant is independent of $A$. Selecting $A$ sufficiently large closes the induction.
\end{proof}

Rearranging \eqref{eq:boundOnINov}, we see that if $C$ is a set of $n$ circles and $\Lambda$ is a set of 
pairwise non-overlapping $k$-rich lenses in $C$, then if $n$ is sufficiently large ($n>8A^3$ will suffice), 
then $|\Lambda|=O \left( \frac{n^2}{k^4} + \frac{n}{k}\right)$.
Thus for all values of $n=|C|$, we have
\begin{equation}\label{eq:bdINovGK}
  \deg(\Lambda) = |\Inov(\Lambda)| = O\left( \frac{n^2}{k^3} + n\right).
\end{equation} 
In particular, Theorem \ref{krich52-lens} is true when $k\geq n^{1/3}\log^{-2/3}n$. 

In the next two sections, we will prove Theorem \ref{krich52-lens} when $2<k<n^{1/3}$,
and also give a second proof of a slightly weaker bound. Note that Theorem \ref{krich52-lens} 
consists of two statements: a bound on $|\Lambda|$ and a bound on $\deg(\Lambda)$. The second 
statement immediately implies the first, by dividing the resulting bound by $k$. 
The next lemma shows that the first statement also implies the second.

\begin{lemma} \label{lem:equivalenceAB}
  Suppose that for every set $C$ of circles in the plane and every $k\geq 2$, every set of 
  pairwise disjoint $k$-rich lenses in $C$ has cardinality at most 
  $A \left(\frac{|C|^{3/2}\log {(|C|/k^3)}}{k^{5/2}} + \frac{|C|}{k} \right)$. 
  Then for every set $C$ of circles in the plane, every $k\geq 2$, and every set 
  $\Lambda$ of pairwise disjoint $k$-rich lenses in $C$, we have 
  $\deg(\Lambda)=O\left(\frac{|C|^{3/2}\log{(|C|/k^3)} }{k^{3/2}} + |C| \right)$, 
  where the implicit constant depends only on $A$.
\end{lemma}
\begin{proof}
Let $C$ be a set of $n$ circles in the plane. Let $k\geq 2$ and let $\Lambda$ be a 
set of pairwise disjoint $k$-rich lenses in $C$. If $k\geq n^{1/3}$, then by \eqref{eq:bdINovGK} 
we have $\deg(\Lambda)=O(n)$ and we are done. 

Suppose now that $2 \leq k \leq n^{1/3}$. Let $\Lambda_0 \subset\Lambda$ be the set of lenses 
that are $n^{1/3}$-rich. Let $j_0$ be the smallest integer so that $2^{-j_0}n^{1/3} \leq k$, 
and for each $j=1,\ldots,j_0$, let $\Lambda_j \subset \Lambda \backslash \bigcup_{i=0}^{j-1}\Lambda_j$ 
be the set of lenses that are $2^{-j}n^{1/3}$-rich. By construction, $\Lambda=\bigsqcup_{j=0}^{j_0}\Lambda_j$, 
and for each index $1\leq j\leq j_0$, the lenses in $\Lambda_j$ have degree between $2^{-j}n^{1/3}$ 
and $2^{-j+1}n^{1/3}$. Thus
\begin{equation*}
  \begin{split}
    \deg(\Lambda) & = \deg(\Lambda_0)+\sum_{j=1}^{j_0}\deg(\Lambda_j)\\
    &\leq \deg(\Lambda_0) + \sum_{j=1}^{j_0}(2^{-j+1}n^{1/3})|\Lambda_j|\\
    &\leq O(n) + \sum_{j=1}^{j_0}(2^{-j+1}n^{1/3})\big( \frac{A n^{3/2}\log {2^{3j} } }{( 2^{-j}n^{1/3} )^{5/2}}\big)\\
    & = O(n) + O\left( 2^{\frac{3}{2}j_0} \cdot \frac{n^{3/2} \log {2^{3j_0}} }{n^{1/2}}\right) \\
    & = O\left(\frac{n^{3/2}\log {(n/k^3)}}{k^{3/2}}+ n\right) ,
  \end{split}
\end{equation*}
where the implicit constant depends on $A$. In the third line we used \eqref{eq:bdINovGK} to bound $\deg(\Lambda_0)=O(n)$.
\end{proof}

\begin{remark}\label{rem:disjointPtsAssumption}
Recall that at the beginning of this section, we added the assumption that no two distinct lenses 
in $\Lambda$ share the same pair $\{p,q\}$ of endpoints. We can now explain why this assumption is harmless. 
Indeed, let $C$ be a set of circles and let $\Lambda$ be a set of pairwise non-overlapping $k$-rich lenses in $C$. 
Let $\Lambda'$ be the set of lenses formed by ``merging'' all lenses in $\Lambda$ that share common
endpoints, i.e., if $\lambda_{p,q}(C')$ and $\lambda_{p,q}(C'')$ are $k$-rich lenses in $\Lambda$, 
then $\lambda_{p,q}(C'\sqcup C'')$ will be an element of $\Lambda'$. While $|\Lambda'|$ might be 
smaller than $|\Lambda|$, we have $\deg(\Lambda')=\deg(\Lambda)$. To summarize: if we can prove that 
every set of $k$-rich lenses with distinct pairs of endpoints has cardinality 
$O\left(\frac{|C|^{3/2}\log {(|C|/k^3)}}{k^{5/2}} + \frac{|C|}{k} \right)$, 
then this implies that every set $\Lambda'$ of $k$-rich lenses with distinct endpoints 
has degree $\deg(\Lambda') = O\left(\frac{|C|^{3/2}\log{(|C|/k^3)} }{k^{3/2}} + |C| \right)$. 
This implies the same bound for any set $\Lambda$ of $k$-rich lenses (i.e.,
the distinct endpoint requirement can be dropped).
\end{remark}

\section{First proof of Theorem \ref{krich52-lens}: Reduction to small $k$} \label{sec:first}

Let $C$ be a set of $n$ circles in the plane, and let $\Lambda$ be a set of pairwise non-overlapping 
$k$-rich lenses in $C$. Let $C^*$, $L=L(\Lambda)$, and $\Inov$ be as defined in Section \ref{sec:prelim}, 
and let $2\leq k\leq n^{1/3}$. By Lemma \ref{lem:equivalenceAB}, to prove Theorem \ref{krich52-lens} 
it suffices to show that $|\Lambda| = O\left(\frac{n^{3/2}\log {(n/k^3)}}{k^{3/2}}\right)$. 
Let $\alpha>0$ be a small absolute constant that will be specified below. We will suppose 
that $2/\alpha\leq k \leq \frac{1}{10}n^{1/3}$, since otherwise Theorem \ref{krich52-lens} 
follows from \eqref{eq:bdINovGK}. 

We can assume that $|\Lambda|\geq 16n/k$, since otherwise we are done. This implies that 
$|\Lambda|\geq 100k^2$ 
Let $D = \alpha k$. As in \cite{GK}, we construct a partitioning polynomial 
$f$ of degree $O(D)$, so that each of the $O(D^3)$ cells of $\reals^3\setminus Z(f)$ contains 
at most $n/D^3$ points of $C^*$ (note that some points of $C^*$ might lie on the zero set $Z(f)$).

Let $L'\subset L$ be the set of lines contained in $Z(f)$. By Lemma \ref{lem:incidencesInZeroSet}, 
there is a set $L''\subset L'$ with $|L''|\leq 11 \deg(f)^2=O(\alpha^2k^2)$ so that 
\[
|\Inov \cap I(C^*\cap Z(P), L'\backslash L'')| \leq 2|C^*\cap Z(P)| + D|L'\backslash L''|\leq 2n + \alpha k |L'\backslash L''|.
\]
Recall that $|L|=|\Lambda|\geq 100k^2$, and thus if $\alpha>0$ is chosen sufficiently small then 
$|L''|\leq |L|/4$. Since each line in $L'\backslash L''$ participates in at least $k$ incidences in $\Inov$, we have
\[
|L'\backslash L''| \leq \frac{1}{k}(2n + \alpha k |L'\backslash L''|) \leq \frac{2n}{k} + \alpha|L|\leq \frac{|L|}{4},
\]
where the final inequality follows by choosing $\alpha < 1/8$ and by using 
the assumption that $|\Lambda|\geq 16n/k$. We conclude that $|L\backslash L'|\geq |L|/2$. 
Next, each $\ell\in L\backslash L'$ participates in at least $k$ 
incidences in $\Inov$, at least $k-\deg(f) \geq (1-O(\alpha))k$ of which must be inside the
cells of $\reals^3\setminus Z(f)$. We say an incidence $(c^*,\ell)\in \Inov$ is \emph{lonely}
if $c^*$ is inside a cell of $\reals^3\setminus Z(f)$, and $(c^*,\ell)$ is the only incidence 
in $\Inov$ involving $\ell$ that occurs inside that cell (i.e., there are no other points of 
$C^*$ on $\ell$ inside that cell, so in the primal plane this implies that this configuration 
does not form a lens). Since each $\ell\in L\backslash L'$ intersects at most 
$\deg(f)+1\leq \alpha k+1$ cells, each $\ell\in L\backslash L'$ participates in at least 
$(1-\alpha)k-1$ incidences in $\Inov$ that are not lonely. Let $\Inovp$ be the set 
of incidences $(c^*,\ell)\in \Inov$ where $c^*$ is inside a cell of $\reals^3\setminus Z(f)$, 
and the incidence is not lonely. Then if $\alpha>0$ is selected sufficiently small, we have
\begin{equation}\label{eq:lowerBdInov}
  |\Inovp| \geq |L\backslash L'|(k/2)\geq \frac{1}{4}k|L|.
\end{equation}
On the other hand, Theorem \ref{thm:MT} says that there are $O\big( (n/D^3)^{3/2} \log(n/D^3)\big)$ 
non-lonely incidences inside each cell. Thus
\begin{equation}\label{eq:upperBdInov}
  |\Inovp| = O\left(D^3 \left(\frac{n}{D^3}\right)^{3/2}\log\big(\frac{n}{D^3}\big)\right)=  
  O\left(\frac{n^{3/2}\log{(n/k^3)}}{k^{3/2}}\right) ,
\end{equation}
where the implicit constant depends on $\alpha$.
Combining \eqref{eq:lowerBdInov} and \eqref{eq:upperBdInov}, we conclude that
\[
|L| = O\left( \frac{n^{3/2}\log {(n/k^3)}}{k^{5/2}}\right).
\]
This completes the first proof of Theorem~\ref{krich52-lens}.

\medskip
\noindent{\bf Remark.}
It is an interesting challenge to extend the analysis in this subsection from circles to
more general families of algebraic curves.  
This topic will be discussed in Section \ref{sec:disc}.

\section{Second proof of Theorem \ref{krich52-lens}: Reduction to large $k$}
\label{sec:second}

In this section we prove a slightly weaker version of Theorem~\ref{krich52-lens}
using a different proof technique. We feel that each of the techniques
is interesting in its own right, and that each has the potential of being
extended into different and more general contexts.

Most of the analysis in this section extends to more general algebraic curves, except for 
one (significant) step. We will discuss possible generalizations in Section \ref{sec:disc}.

Define $F(n,k)$ to be the smallest integer with the following property: Let $C$ be a set of 
at most $n$ circles in the plane; let $\Lambda$ be a set of pairwise disjoint $k$-rich lenses in $C$. 
Then $\deg(\Lambda)\leq F(n,k)$. Note that $F(n,1)=\infty$ (i.e., it is undefined for $k=1$),
and, trivially, $F(n,k)=O(n^2)$ for all $k\geq 2$. 
Furthermore, $F(n,k)$ is monotone increasing in $n$ and monotone decreasing in $k$. Abusing notation 
slightly, we extend our definition of $F(n,k)$ to all real numbers $n\geq 1$ and $k\geq 2$ by defining 
$F(n,k) = F(\lceil n\rceil, \lfloor k \rfloor)$. 

In Section \ref{sec:prelim} we proved that $F(n,k)=O(n^2/k^3 + n)$, and in particular 
there is an absolute constant $A_0$ so that, for any $z > 1$,
\begin{equation}\label{startingBd}
  F(k^3 z,k) \leq A_0 k^3 z^2.
\end{equation}
In this section we will establish the following recurrence relation for $F(n,k)$.
\begin{lemma} \label{lem:recurrenceForF}
  For any $D\geq 1$ we have
  \begin{equation}\label{eq:FnkRecurrence}
  F(n,k)\leq A D^3 F(n/D^2, k/3)+ F(AD^2,k/3)+A D^2 n ,
  \end{equation}
  for a suitable absolute constant $A$.
\end{lemma}
Before proving Lemma \ref{lem:recurrenceForF}, we show that it implies 
\begin{equation}\label{eq:boundOfFnk}
  F(n,k) = O\left(\frac{n^{3/2}\log^b (n/k^3)}{k^{3/2}}\right),
\end{equation}
for some constant $b$ and for all $n > k^3$. To show this, we solve the recurrence in the lemma
in several steps. First, given $n$ and $k$, we construct a sequence of real numbers $n_0,n_1,\ldots,n_s = n$,
where $n_0 = k^3 z$, for a suitable value of $z > 1$ (the actual value will be between $\sqrt 2$ and $2$,
and its concrete choice will be given towards the end of the forthcoming analysis), and $n_{j+1} = n_j^2/k^3$,
for $j\ge 0$. That is, $n_j = k^3 z^{2^j}$ for $j\ge 0$, as is easily verified by induction on $j$. 
Since we want $n_s$ to be equal to $n$, we have
$z^{2^j} = n/k^3$. We also define, for each $j$, $D_j := n_j^{1/2}/k^{3/2} = z^{2^{j-1}}$, 
and note that $n_{j+1} = D_j^2 n_j$. The rationale for choosing these sequences will become 
clear as the solution of the recurrence unfolds. We have
\begin{equation} \label{djnj}
\frac{n_{j+1}^{3/2}}{k^{3/2}} = \frac{D_j^3 n_j^{3/2}}{k^{3/2}} = \frac{D_j^2 n_j^2}{k^3} = D_j^2 n_{j+1} .
\end{equation}
Note that $AD_j^2= A n_j / k^3$. For simplicity, we shall suppose that $k\geq A^{1/3}$ and thus $AD_j^2 \leq n_j$
(if this inequality failed then Theorem \ref{krich52-lens} follows from Theorem \ref{thm:MT}, since
$k$ becomes a constant).

We next prove that for each $j\geq 0$ we have
\begin{equation} \label{eq:f}
  F(n_j,3^jk) \leq A_0 z^{1/2}(3A)^j n_j^{3/2}/k^{3/2} , 
\end{equation}
where $A_0$ is the constant from \eqref{startingBd}. The case $j=0$ is precisely \eqref{startingBd}. 
For the induction step, we compute, using Lemma~\ref{lem:recurrenceForF}:
\begin{equation*}
\begin{split}
F(n_{j+1}, 3^{j+1} k) & \leq A D_j^3 F(n_{j+1}/D_j^2, 3^j k) + F(AD_j^2, 3^j k)+A D_j^2 n_{j+1}\\
& \leq  2A D_j^3 F(n_{j}, 3^j k) + A D_j^2 n_{j+1}\\
& \leq 2A D_j^3 (A_0 z^{1/2}) (3A)^j n_j^{3/2} / k^{3/2} + A D_j^2 n_{j+1}\\
& \le 2\cdot 3^j (A_0 z^{1/2}) A^{j+1} n_{j+1}^{3/2} / k^{3/2} + A D_j^2 n_{j+1}\\
& \leq A_0 z^{1/2} (3A)^{j+1} n_{j+1}^{3/2} / k^{3/2} ,
\end{split}
\end{equation*}
where in the last inequality we used the equality (\ref{djnj}), namely $D_j^2 n_{j+1} = n_{j+1}^{3/2}/k^{3/2}$.

Thus if $n > k^3$, we can find $z > 1$ and $s$ so that $n_s = k^3 z^{2^s} = n$.
That is, $2^s = \frac{\log (n/k^3)}{\log z}$ or $s = \log\log(n/k^3) - \log\log z$.
We use (\ref{eq:f}), with $j=s$, $n_j = n$ and with replacing $k$ by $k/3^s$, and obtain
\[
F(n,k) \leq A_0 z^{1/2}(3^{5/2}A)^s \frac{n^{3/2}}{k^{3/2}} .
\]
Putting $B := 3^{5/2}A$ and $b := \log B$, we get
\[
(3^{5/2}A)^s = B^s = (2^s)^b = 
\left( \frac{\log (n/k^3)}{\log z} \right)^b = 
\frac{\log^b (n/k^3)}{\log^b z} ,
\]
and hence
\[
F(n,k) \leq A_0 z^{1/2} B^s \frac{n^{3/2}}{k^{3/2}} =
\frac{A_0 z^{1/2} }{\log^b z} \cdot
\frac{n^{3/2}\log^b (n/k^3)}{k^{3/2}} .
\]
It remains to determine the value of $z$. Put $z_j = (n/k^3)^{1/2^j}$, for $j\ge 0$.
This sequence converges to $1$ and satisfies $z_j = \sqrt{z_{j-1}}$ for each $j$.
We take $s$ to be that (unique) value of $j$ for which $\sqrt{2} < z_j \le 2$
(for such a $z$ to exist we need to assume that $n > k^3\sqrt{2}$). We then have
\[
\frac{A_0 z^{1/2} }{\log^b z} \le 
A_1 := \frac{A_0\sqrt{2} }{\log^b \sqrt{2}} , \qquad\text{and so}\qquad
F(n,k) \leq A_1 \frac{n^{3/2}\log^b (n/k^3)}{k^{3/2}} .
\]
This establishes (\ref{eq:boundOfFnk}), and leaves us with the task of
proving Lemma \ref{lem:recurrenceForF}. 

\begin{proof}[Proof of Lemma \ref{lem:recurrenceForF}]
Let $C$ be a set of $n$ circles in the plane and let $\Lambda$ be a set of pairwise non-overlapping 
$k$-rich lenses in $C$. Following the technique of Ellenberg, Solymosi, and Zahl \cite{ESZ}, 
for each circle $c\in C$ with defining polynomial $g$ (i.e., $c = Z(g)$), consider the algebraic variety
\[
\{ (x,y,z) \in \reals^3 \mid g(x,y) = 0,\ z \partial_y g(x,y) + \partial_x g(x,y) = 0\}.
\] 
As discussed in \cite{ESZ} (see also \cite{SZ}), this variety is a union of three irreducible 
curves in $\reals^3$, two of which are vertical lines (one above each of the points in $c$ 
where the circle has infinite slope). Define $\gamma(c)\subset\reals^3$ to be the irreducible 
component that is not a vertical line. If $(x,y)\in c$ is a point where $c$ has finite slope, 
then $(x,y,z)\in\gamma(c)$ if and only if $c$ has slope $z$ at $(x,y)$. In particular, if 
$\lambda_{p,q}(C')$ is a lens and if $c\in C'$, then the (shorter) arc $\beta\subset c$ with 
endpoints $p$ and $q$ lifts to a curve segment $\gamma(\beta)\subset \gamma(c)$. We will call 
this curve segment the lifted arc of $c$ corresponding to the lens $\lambda$. 

For a set of circles $C$, define $\gamma(C) = \{\gamma(c)\mid c\in C\}$. Let $\Lambda_{p,q}(C')$ 
be a lens in $C$, and suppose that none of the circles $c\in C'$ have infinite slope at the point $p$ or $q$ 
(this is a harmless assumption, since at most two circles containing $p$ and $q$ can have infinite 
slope at $p$ or $q$). Let $\ell_p,\ell_q\subset\reals^3$ be vertical lines passing through $(p,0)$ 
and $(q,0)$ respectively. Then each of the curves in $\gamma(C')$ intersect $\ell_p$ and $\ell_q$. 
Furthermore, each of the intersection points $\{\gamma(c)\cap \ell_p \mid c \in C'\}$ are distinct, 
and similarly for $\ell_q$. Define $z(\gamma(c)\cap \ell_p)$ to be the $z$-coordinate of 
$\gamma(c)\cap \ell_p$. The curves in $\gamma(C')$ have the following property:

\medskip
\noindent \textbf{Order Reversal Property}. If we order the curves $c_1,\ldots,c_m\in C'$ so that 
\[
z(\gamma(c_1)\cap \ell_p)<z(\gamma(c_2)\cap \ell_p)<\cdots<z(\gamma(c_m)\cap \ell_p),
\]
then
\[
z(\gamma(c_1)\cap \ell_q)>z(\gamma(c_2)\cap \ell_q)>\cdots>z(\gamma(c_m)\cap \ell_q),
\]
i.e., the order on $C'$ given by the $z$-coordinates of $\gamma(c)\cap\ell_p$ is precisely 
the reverse of the order given by $\gamma(c)\cap\ell_q$. See Figure \ref{multicyc}.

\begin{figure}[htb]
  \begin{center}
    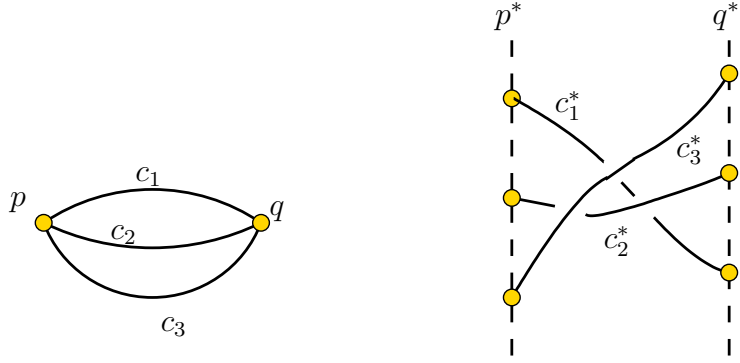
    \caption{A lens is lifted to a multi-$2$-cycle in three dimensions.} 
    \label{multicyc}
  \end{center}
\end{figure}

We employ the approach of Aronov and Sharir~\cite{ArS}, as detailed in
Sharir and Zahl~\cite{SZ}, with some modifications, as follows. We construct a partitioning 
polynomial $f$, of degree $O(D)$, so that we have  $O(D^3)$ open connected cells of 
$\reals^3\setminus Z(f)$, and at most $n/D^2$ curves from $\Gamma(C)$ intersect each cell. 
The existence of such a partitioning polynomial was established in Guth~\cite{Guth}. 
For each cell $O$ of $\reals^3\backslash Z(f)$, define 
$C_O = \{ c\in C\mid \gamma(c) \cap O \neq\emptyset\}$. 

For a cell $O\subset \reals^3\setminus Z(f)$, we say that a lens $\lambda=\lambda_{p,q}(C')\in\Lambda$ 
is \emph{preserved} within $O$ if for at least $\deg(\lambda)/3=|C'|/3$ circles $c\in C'$, the lifted 
arc of $c$ corresponding to the lens $\lambda$ is contained in $O$. In particular, if $\lambda_{p,q}(C')$ 
is preserved within $O$, then $|C'\cap C_O|\geq |C'|/3\geq k/3$. Thus for each cell $O$, we have
\[
\sum \deg(\lambda_{p,q}(C') ) \leq 3 \sum|C'\cap C_O| \leq 3 F(n/D^2,k/3),
\] 
where the sum is taken over all lenses $\lambda_{p,q}(C')$ that are preserved within $O$. 
Summing over all cells $O$, we conclude that
\begin{equation} \label{eq:bd1}
\sum_{\lambda\ \textrm{preserved within a cell}}\deg(\lambda)=O(D^3)F(n/D^2,k/3).
\end{equation} 
If a lens is not preserved within any cell, we say that it is \emph{disrupted} by $Z(f)$. 
It remains to bound the sum of the degrees of the disrupted lenses. The arguments here 
are very similar to those in \cite{SZ}, so we just sketch them briefly and highlight 
the key differences. First, for each $(x,y,z)\in\reals^3$, define $h(x,y,z)$ to be the number 
of intersections between $Z(f)$ and the infinite ray $\{(x,y,t)\mid t>z\}$. This quantity 
is finite (indeed bounded by $\deg f$) unless the vertical line passing through $(x,y,z)$ 
is contained in $Z(f)$. Following the arguments in \cite{ArS,SZ}, there is a polynomial 
$g\in\reals[x,y,z]$ of degree $O(D^2)$ with the following properties.
\begin{itemize}
\item 
  $h$ is constant on each connected component of $\reals^3\backslash \big(Z(f) \cup Z(g)\big)$.
\item 
  $h$ is constant on $Z(f) \backslash Z(g)$.
\item 
  $g(x,y,z)$ is independent of $z$, i.e., $g(x,y,z) = \tilde g(x,y)$ for some polynomial 
  $\tilde g(x,y)\in\reals[x,y]$.  
\item 
  If $Q(x,y,z)$ is an irreducible component of $f$ that is independent of $z$, then $Q$ divides $g$, 
  i.e., $Q$ is also an irreducible component of $g$.
\end{itemize} 
In brief, the polynomial $g(x,y,z) = \tilde g(x,y)$ is constructed by computing the resultant of $f$ and $\partial_z f$; see \cite{ArS,SZ} for details.

We say that a lens $\lambda_{p,q}(C')$ is preserved by $Z(g)$ if at least $\deg(\lambda)/3=|C'|/3$ 
of the curves from $\gamma(C')$ are contained in $Z(g)$. Recall that $g(x,y,z) = \tilde g(x,y)$, 
and thus if $\gamma(c)\subset Z(g)$, we must have $c\subset Z(\tilde g)$. In particular, at most 
$\deg(g) = \deg(\tilde g) = O(D^2)$ circles from $C$ can be contained in $Z(g)$. Arguing as before, we conclude that
\begin{equation} \label{eq:bd2}
\sum_{\lambda\ \textrm{preserved by}\ Z(g)}\deg(\lambda)\leq 3F(O(D^2), k/3).
\end{equation} 
It remains to bound the sum of the degrees of the lenses that are disrupted by $Z(f)$ and not 
preserved by $Z(g)$. We claim that if $\lambda_{p,q}(C')$ is such a lens, then there are at least 
$|C'|/3$ circles $c\in C'$ so that the lifted arc of $c$ corresponding to the lens $\lambda$ properly 
intersects $Z(f)$ or properly intersects $Z(g)$. Once this has been established we are done, since 
the number of such proper intersections is at most $n(\deg f + \deg g) = O(D^2n)$, and since the 
lenses are pairwise non-overlapping, each such intersection is counted towards at most one lens. 

To verify this claim, let $\lambda_{p,q}(C')$ be a lens that is disrupted by $Z(f)$ and not 
preserved by $Z(g)$. We will divide our argument into cases. 

\medskip

\noindent{\bf Case 1}: At least $|C'|/3$ of the lifted circles in $C$ are contained in $Z(f)$ but not contained in $Z(g)$.\\
Enumerate the circles contained in $Z(f)$ but not contained in $Z(g)$ as $c_1,\ldots,c_w$, for some
$w\geq |C'|/3$, so that $z(\gamma(c_1)\cap \ell_p) < \cdots < z(\gamma(c_w)\cap \ell_p)$. 
Since each circle $c_i$ is contained in $Z(f)$ but not contained in $Z(g)$, we have
\[
h(\gamma(c_1)\cap \ell_p)< \cdots < (\gamma(c_w)\cap \ell_p).
\]
However, by the Order Reversal Property, we have $z(\gamma(c_1)\cap \ell_p) > \cdots > z(\gamma(c_w)\cap \ell_p)$, and thus
\[
h(\gamma(c_1)\cap \ell_p)> \cdots > (\gamma(c_w)\cap \ell_p).
\]
Since $h$ is constant on $Z(f) \backslash Z(g)$, we conclude that for all but at most 
one index $i$, the lifted arc of $c_i$ corresponding to the lens $\lambda$ intersects $Z(g)$. 

\medskip

\noindent{\bf Case 2}: At least $2|C'|/3$ of the circles in $C'$ are not contained in $Z(f)$ or $Z(g)$.\\
Let $c_1,\ldots,c_w$, for some $w \geq 2|C'|/3$, be circles in $C'$ whose lifted curve 
is not contained in $Z(f)$ or $Z(g)$. 
For each index $i=1,\ldots,w$, let $\beta_i$ be the (shorter) arc of $c_i$ with endpoints $p$ and $q$. 
Let $v$ be the number of arcs $\gamma(\beta_i)$ that properly intersect $Z(f)$; 
if $v\geq |C'|/3$ then we are done. If not, then at least $w-v$ of the arcs $\gamma(\beta_i)$ 
are contained inside a cell of $Z(f)$ (though different arcs might be contained inside different cells). 
But an argument analogous to the one above shows that all of the arcs in all
but one of the cells must properly intersect $Z(g)$ Since no cell contains more 
than $|C'|/3$ of the lifted arcs, at least $w-v-|C'|/3$ of the lifted arcs must 
properly intersect $Z(g)$. We conclude that $v$ arcs properly intersect $Z(f)$ 
and at least $w-v-|C'|/3\geq |C'|/3 - v$ arcs properly intersect $Z(g)$. 
Thus at least $|C'|/3$ arcs properly intersect either $Z(f)$ or $Z(g)$. 

Combining the bounds in (\ref{eq:bd1}), (\ref{eq:bd2}), adding the overhead $O(D^2n)$,
and making the constants in the $O(\cdot)$ notation explicit, bounding all of them
by the same constant $A$, we obtain the recurrence asserted in the lemma.
\end{proof}

\section{Point-circle and lens-circle incidence bounds}
\label{sec:incidence_bounds}
\subsection{Point-circle incidence bounds}

We can use Theorem~\ref{krich52-lens} to bound the number of incidences 
between $m$ points and $n$ circles in the plane. 
As it turns out, the bound that we get is the same as the best 
known bound due to Agarwal et al.~\cite{ANPPSS} (and to \cite{MT}). 
We describe the derivation nonetheless, as an illustration of the power of
Theorem~\ref{krich52-lens}.

Let $P$ be a set of $m$ points and let $C$ be a set of $n$ circles. We fix a 
parameter $k$, to be determined below, and use a modified variant of Sz\'ekely's 
technique~\cite{Sze}. We first construct a graph $G$ whose vertices are the points
of $P$, and whose edges connect pairs of consecutive points along each circle of $C$.
Some edges of $G$ form $k$-rich lenses, and we observe that these lenses are pairwise 
non-overlapping. Let $\Lambda$ denote the set of these lenses. We split $G$ into two
subgraphs $G_0$ and $G_1$, where $G_1$ consists of all the edges in the lenses of $\Lambda$
and $G_0$ consists of all the remaining edges. 

By Theorem~\ref{krich52-lens}, the number of edges of $G_1$ is
${\displaystyle O\left( \frac{n^{3/2}\log (n/k^3)}{k^{3/2}} + n \right)}$.

The number $E_0(c)$ of edges of $G_0$ along a circle $c$ is $|N_c| - E_1(c)$, 
where $N_c = P\cap c$ and $E_1(c)$ is the number of edges of $G_1$ along $c$.
Note that the multiplicity of each edge of $G_0$ is smaller than $k$.
An upper bound on the number of edges of $G_0$ then follows from a variant of 
Sz\'ekely's technique (see Theorem 7 of \cite{Sze}), which takes into account 
the maximum multiplicity of an edge in the graph (which is smaller than $k$). 
Concretely, denoting by $|G_0|$ (resp., $|G_1|$) the number of edges of $G_0$ 
(resp., $G_1$), we have 
\[
|G_0| = O\left( k^{1/3}m^{2/3}n^{2/3} + km + |G_1| \right) = 
O\left( \frac{n^{3/2}\log (n/k^3)}{k^{3/2}} + k^{1/3}m^{2/3}n^{2/3} + km + n \right) 
\]
(this expression actually bounds the size of the whole graph, namely $|G| = |G_0|+|G_1|$),
and we balance the first two terms by choosing  $k = n^{5/11}(\log (n/k^3))^{6/11}/m^{4/11}$. 
This is meaningful when $k\ge 1$, which holds when $m \le n^{5/4}\log^{3/2}n$, 
which is indeed the interesting range. For larger values of $m$,
we take $k=2$ and get the bound $O(m^{2/3}n^{2/3} + m + n^{3/2}\log n)$,
which is dominated by $O(m^{2/3}n^{2/3} + m)$. The bound then becomes (see \cite{ANPPSS}) 
\[
O\left( m^{2/3}n^{2/3} + m^{6/11}n^{9/11}\log^{2/11} (m^3/n) + m + n \right).
\]
Note that the bound is meaningful only for $m > n^{1/3}$. For smaller values of $m$,
the bound becomes $O(n)$. The logarithmic factor provides a `smooth' transition from the
above bound to the linear bound as $m\downarrow n^{1/3}$.

\subsection{Circle-lens incidence bounds}

We can apply the bounds in Theorem~\ref{krich52-lens} to obtain an upper bound
on the number of incidences between $m$ pairwise non-overlapping lenses 
and $n$ circles, where a lens $\lambda$ is said to be incident to a circle 
$c$ if $c$ participates in $\lambda$. To do so, let $\Lambda$ be the given
set of $m$ lenses (which are not necessarily rich). Set 
${\displaystyle k:= \frac{n^{3/5}\log^{2/5}(n/k^3)}{m^{2/5}}}$. (Note that 
$k = \Omega(1)$ since we always have $m = O(n^{3/2}\log n)$~\cite{ANPPSS,MT}.)
The $k$-poor lenses of $\Lambda$ contribute at most 
$km = m^{3/5}n^{3/5}\log^{2/5}(n/k^3)$ incidences. The $k$-rich lenses
contribute, by Theorem~\ref{krich52-lens},
\[
O\left( \frac{n^{3/2}\log (n/k^3)}{k^{3/2}} + n \right) =
O\left( m^{3/5}n^{3/5}\log^{2/5}(n/k^3) + n \right)
\]
incidences. Since $\log(n/k^3) = O(\log (m^3/n^2))$, we thus obtain:

\begin{theorem} \label{lens-circ}
Let $\Lambda$ be a family of $m$ pairwise non-overlapping lenses in an 
arrangement of $n$ circles in the plane. Then the number of 
incidences between the lenses of $\Lambda$ and the circles of $C$ is 
${\displaystyle O\left( m^{3/5}n^{3/5}\log^{2/5}(m^3/n^2) + n \right)}$.
\end{theorem}

\medskip
\noindent{\bf Remark.}
Aside from the $\log$ factor, this bound generalizes the recent result of Sharir and Zlydenko~\cite{ShZl}
(see also Sharir, Solomon, and Zlydenko \cite{ShSoZl}) on incidences between so-called 
directed points and circles. A directed point is a pair $(p,u)$ 
where $p$ is a point in the plane and $u$ is a direction, and 
$(p,u)$ is incident to a circle $c$ if $p\in c$ and $u$ is the 
direction of the tangent to $c$ at $p$. The bound in \cite{ShSoZl,ShZl} 
is $O(m^{3/5}n^{3/5}+m+n)$ which is similar, albeit slightly sharper, than the bound in
Theorem~\ref{lens-circ}. The two setups are indeed related, as a
directed point of degree at least two is a limiting case of a lens, and the resulting 
infinitesimal limit lenses are clearly pairwise non-overlapping.
The novelty in Theorem~\ref{lens-circ} is that lenses are 
$4$-parameterizable, that is, each lens is specified by four 
real parameters (the coordinates of its vertices $p$, $q$), whereas
directed points are $3$-parameterizable. This makes the analysis
in \cite{ShSoZl,ShZl} inapplicable to the case of lenses, and yet
the bound is more or less preserved.

\section{Discussion} \label{sec:disc}

Each of the two proofs of the main result, given in 
Sections~\ref{sec:first} and \ref{sec:second}, can be extended
to more general contexts, provided that certain key properties
can be established, or alternatively are assumed. In this section we
discuss such possible extensions, and then summarize the state of
affairs developed in this paper.

\subparagraph*{First proof.}

We offer a few informal comments on a possible approach to extending Theorem \ref{krich52-lens} 
to more general plane curves. First, we need to assume that the curves in our
family $C$ are $3$-parameterizable, so that we can represent them as points
in a dual $3$-space, and also that they are algebraic of some constant degree. 
Each point $p\in\reals^2$ then becomes a two-dimensional surface $p^*$, consisting 
of the points in $\reals^3$ whose corresponding curves contain $p$. Then a lens 
with endpoints $p$, $q$ becomes the curve $\ell_{p,q} = p^*\cap q^*$ (we ignore in 
this informal discussion various issues involving degeneracies and various assumptions 
that one might need to impose).

We can then apply the same partitioning argument. Inside each cell, we use the (slightly weaker) 
bound $O(n^{3/2}{\rm polylog}(n))$, due to Sharir and Zahl~\cite{SZ}, 
on the number of lenses formed by a set of bounded-degree algebraic curves.

The main difference is in handling points and curves that lie on the zero set
of the partitioning polynomial. The preceding analysis strongly relied on 
Lemma \ref{no3coplanar}, which requires that the curves in $C$ be circles. 
This in turn allowed us to control the number of incidences occurring on the 
zero-set $Z(f)$ of the partitioning polynomial. With an analogue of Lemma \ref{no3coplanar} 
for more general curves, it seems plausible that the rest of the argument will work with standard modifications. 

\subparagraph*{Second proof.}

We remark that Lemma \ref{lem:recurrenceForF} holds with almost no modification if 
the circles in $C$ are replaced by arbitrary (bounded degree, algebraic) curves. 
Indeed, the only important difference is that the Order Reversal Property might 
not be true, but the dichotomy that a lens must either be preserved within a cell 
or disrupted by $Z(f)$ remains true, and the bound on the number of lenses that are 
disrupted by $Z(f)$ and not preserved by $Z(g)$ also remains true. Thus the only 
obstruction to extending Theorem \ref{krich52-lens} to more general curves is that 
the estimate $F(n,n^{1/3})=O(n)$ (or, more precisely, $F(nz,n^{1/3})=O(nz^2)$ 
for $z>1$), which serves as the base case of the induction, 
might not be true. We conjecture that for other classes of curves, an estimate of 
the form $F(n,n^b)=O(n)$ should hold (where $b>0$ depends on the class of curves). 
As $b$ becomes larger, the corresponding analogue of Theorem~\ref{lem:recurrenceForF} 
becomes weaker.


\begin{thebibliography}{}

\bibitem{ANPPSS}
P. Agarwal, E. Nevo, J. Pach, R. Pinchasi, M. Sharir and S. Smorodinsky,
Lenses in arrangements of pseudocircles and their applications,
{\it J. ACM} 51 (2004), 139--186.

\bibitem{AgPa}
P. Agarwal and J. Pach, 
{\it Combinatorial Geometry},
Wiley Interscience, New York, 1995.

\bibitem{ArS:lens}
B. Aronov and M. Sharir, 
Cutting circles into pseudo-segments and improved bounds for incidences,
{\it Discrete Comput. Geom.} 28 (2002), 475--490.

\bibitem{ArS}
B. Aronov and M. Sharir, 
Almost tight bounds for eliminating depth cycles in three dimensions,
{\it Discrete Comput. Geom.} 59 (2018), 725--741.
Also in arXiv:1512.00358.

\bibitem{ESZ}
J. Ellenberg, J. Solymosi and J. Zahl,
New bounds on curve tangencies and orthogonalities,
{\it Discrete Analysis} (2016), paper 18.
Also in arXiv:1509:05821v4.

\bibitem{Guth}
L.~Guth,
Polynomial partitioning for a set of varieties, 
{\it Math. Proc. Camb. Phil. Soc.} 159 (2015), 459--469.
Also in arXiv:1410.8871.

\bibitem{GK}
L.\ Guth and N.\ H.\ Katz,
On the Erd{\H o}s distinct distances problem in the plane,
{\it Annals Math.} 181 (2015), 155--190.
Also in arXiv:1011.4105.

\bibitem{MT}
A. Marcus and G. Tardos,
Intersection reverse sequences and geometric applications,
{\it J. Combinat. Theory} Ser.~A, 113 (2006), 675--691.

\bibitem{ST}
E.~Szemer{\'e}di and W.~T.~Trotter, Jr.
Extremal problems in discrete geometry,
{\it Combinatorica} 3 (1983), 381--392.

\bibitem{ShSoZl}
M. Sharir, N. Solomon and O. Zlydenko,
Incidences with curves with almost two degrees of freedom,
In arXiv:2003.02190v2.

\bibitem{SZ}
M. Sharir and J. Zahl,
Cutting algebraic curves into pseudo-segments and applications,
{\it J. Combinat. Theory} Ser. A 150 (2017), 1--35. 
Also in arXiv:1604.07877.

\bibitem{ShZl}
M. Sharir and O. Zlydenko,
Incidences with curves with almost two degrees of freedom,
{\it Proc. 36th Sympos. on Computational Geometry} (2020), to appear.

\bibitem{Sze}
L. Sz\'ekely,
Crossing numbers and hard Erd\H{o}s problems in discrete geometry,
{\it Combinat. Probab. Comput.} 11 (1993), 1--10.

\end{thebibliography}
\end{document}